\documentclass[11pt,a4paper]{amsart}
\usepackage[a4paper, total={15cm, 23cm}]{geometry}
\usepackage{amsmath,amssymb,amscd,amsthm,amsxtra,array,soul}
\usepackage[usenames,dvipsnames]{xcolor}
\usepackage[
colorlinks=true,
linkcolor=blue,
citecolor=blue,
urlcolor=blue,
]{hyperref}

\newtheorem{theorem}{Theorem}[section]
\newtheorem{lemma}[theorem]{Lemma}

\newtheorem{corollary}[theorem]{Corollary}

\theoremstyle{definition}
\newtheorem{definition}[theorem]{Definition}

\theoremstyle{remark}
\newtheorem{remark}[theorem]{Remark}
\numberwithin{equation}{section}

\newcommand{\RN }{\mathbb R^n}
\newcommand{\BMO}[0]{\operatorname{BMO}}
\newcommand{\abs}[1]{|#1|}

\newcommand{\Norm}[2]{\|#1\|_{#2}}

\newcommand{\strt}[1]{\rule{0pt}{#1}}

\newcommand{\norm}[1]{\mbox{$\left\| #1 \right\|$}}

\def\Xint#1{\mathchoice
  {\XXint\displaystyle\textstyle{#1}}%
  {\XXint\textstyle\scriptstyle{#1}}%
  {\XXint\scriptstyle\scriptscriptstyle{#1}}%
  {\XXint\scriptscriptstyle\scriptscriptstyle{#1}}%
  \!\int}
\def\XXint#1#2#3{{\setbox0=\hbox{$#1{#2#3}{\int}$}
    \vcenter{\hbox{$#2#3$}}\kern-.5\wd0}}

\def\avgint{\Xint-}
\newcommand{\vertiii}[1]{{\left\vert\kern-0.25ex\left\vert\kern-0.25ex\left\vert #1 
    \right\vert\kern-0.25ex\right\vert\kern-0.25ex\right\vert}}

\numberwithin{equation}{section}

\begin{document}

\title{A note on generalized Fujii-Wilson conditions and $\BMO$ spaces}

\author[S. Ombrosi]{Sheldy Ombrosi}
\address[Sheldy Ombrosi] {Instituto de Matem\'atica de Bah\'{\i}a Blanca (INMABB), Departamento de Matem\'atica, Universidad Nacional del Sur (UNS) - CONICET, Av. Alem 1253, Bah\'{\i}a Blanca, Argentina}
\email{sombrosi@uns.edu.ar}

\author[C. P\'erez]{Carlos P\'erez}
\address[Carlos P\'erez]{ Department of Mathematics, University of the Basque Country, IKERBASQUE 
(Basque Foundation for Science) and
BCAM \textendash  Basque Center for Applied Mathematics, Bilbao, Spain}
\email{cperez@bcamath.org}

\author[E. Rela]{Ezequiel Rela}
\address[Ezequiel Rela]{Departamento de Matem\'atica,
Facultad de Ciencias Exactas y Naturales, Universidad de Buenos Aires, Ciudad Universitaria Pabell\'on I, Buenos Aires 1428 Capital Federal Argentina} \email{erela@dm.uba.ar}

\author[I.P. Rivera-R\'{\i}os]{Israel P. Rivera-R\'{\i}os}
\address[Israel P. Rivera-R\'{\i}os] {Instituto de Matem\'atica de Bah\'{\i}a Blanca (INMABB), Departamento de Matem\'atica, Universidad Nacional del Sur (UNS) - CONICET, Av. Alem 1253, Bah\'{\i}a Blanca, Argentina}
\email{israel.rivera@uns.edu.ar}

\thanks{C. P. This work was supported by the Spanish Ministry of Economy and Competitiveness, MTM2017-82160-C2-2-P and SEV-2017-0718}
\thanks{E. R. is partially supported by grants UBACyT 20020170200057BA and PICT-2015-3675.}
\thanks{S. O. and I. P. R.-R. are supported by grant  PIP (CONICET) 11220130100329CO}



\begin{abstract}
In this note we generalize the definition of Fujii-Wilson condition providing quantitative characterizations of some interesting classes of weights, such as $A_\infty$, $A_\infty^{weak}$ and $C_p$, in terms of  $\BMO$ type spaces suited to them. We will provide as well some self improvement properties for some of those generalized $\BMO$ spaces and some quantitative estimates for Bloom's $\BMO$ type spaces.
\end{abstract}

\maketitle

\section{Introduction and main results}

Given a weight $v$, namely, non-negative locally integrable function in $\RN$,  and a functional $Y:\mathcal{Q}\rightarrow (0,\infty)$ defined over the family of all cubes  in $\RN$ with sides parallel to the axes, we define the class of functions $\BMO_{v,Y}$ by  
$$
\BMO_{v,Y}=  \{f \in L_{loc}^{1}(\RN):  \Norm{f}{\BMO_{v,Y}}<\infty \}
$$
where 
\begin{equation*}
  \Norm{f}{\BMO_{v,Y}}:=\sup_Q\frac{1}{Y(Q)}\int_Q\abs{f-f_{Q} }v<\infty,
\end{equation*}
and, as usual, $f_Q= \frac{1}{|Q|}\int_Q f$ denotes the average of $f$ over $Q$. 
In the case that $Y(Q)=v(Q)$ for every cube $Q$, a classical result due to Muckenhoupt and Wheeden in \cite{MW} asserts that
$$
\BMO = \BMO_{v,v}
$$
holds whenever $v\in A_{\infty}$. 
  Also  the case of $Y(Q)=w(Q)$ for some weight  $w$ and $v=1$  was considered in \cite{MW} and independently by J. Garc\'{\i}a-Cuerva  \cite{GC} in the context of Hardy spaces\footnote{Results involving that space appear in an abstract of that author in Notices of the AMS Feb 1974 p.A-309.}. Later on, S. Bloom's \cite{Bloom} also considered this special case in the context of commutators and used the notation  $\BMO_w$ to denote the space $\BMO_{1,w}$. We remit the reader to  \cite{GCHST, HLW, LORRAdv, LORR, HComm2, AMPRR} for the latest advances and related results in that direction. The unweighted case $v=1$ and $Y(Q)=|Q|$ corresponds, obviously, to the classical $\BMO$ space of John-Nirenberg \cite{JN} and in that case we shall drop the subscripts. 

More recently it was established that $I:\BMO \hookrightarrow \BMO_{w,w}$ is bounded with norm at most $c_n[w]_{A_{\infty}}$, namely 
\begin{equation}\label{BMO-BMOw}
\Norm{f}{\BMO_{w,w}} \leq c_n\, [w]_{A_{\infty}}\,\Norm{f}{\BMO}
\end{equation}
where  $[w]_{A_{\infty}}$ denotes the Fujii-Wilson constant defined by
\begin{equation}\label{FW}
[w]_{A_\infty}:=\sup_Q\frac{1}{w(Q)}\int_Q M(w\chi_{\strt{1.7ex} Q}),
\end{equation}
and the dependence on $[w]_{A_\infty}$ is sharp. Definition \eqref{FW} was introduced in \cite{HP} where estimate \eqref{BMO-BMOw} was obtained combining the classical John-Nirenberg theorem and the optimal reverse H\"older inequality obtained as well in \cite{HP}, namely,  if $w \in A_{\infty}$ and if 
$$r(w):=1+\frac{1}{ \tau_n \,[w]_{A_{\infty}}}$$ 
then,
\begin{equation} \label{RHI}
  \Big(\avgint_Q w^{r(w)}\Big)^{\frac{1}{r(w)} }\leq 2\avgint_Q w.
\end{equation}
with $\tau_n$ a dimensional constant that we may take to be $\tau_n \approx 2^{n}$.

Up until now it was not known whether \eqref{BMO-BMOw} would still be true for general weights, in other words, 
whether the emdedding $I:\BMO \hookrightarrow \BMO_{w,w}$ would hold for weights beyond the $A_{\infty}$ class. Our first result gives a negative answer to that question  providing a new quantitative characterization of $A_\infty$ in terms of $[w]_{A_\infty}$ and the $\BMO$ class. We recommend \cite{DMRO-Ainfty} for a detailed account on characterizations of $A_\infty$. Actually, we are going to present a more general result that encloses the aforementioned results as particular cases. For that purpose we introduce here the following variation of the Fujii-Wilson $A_\infty$ constant \eqref{FW}.

\begin{definition}
Let $v$ be a weight and  $Y:\mathcal{Q}\rightarrow (0,\infty)$ a functional defined over the family of all cubes in $\RN$. We define
\begin{equation}\label{eq:twoweight-Ainfty}
[v]_{A_{\infty,Y}}=\sup_{Q}\frac{1}{Y(Q)}\int_QM(v\chi_{\strt{1.7ex} Q}).
\end{equation}
When the supremum above is finite, we say that  $v\in A_{\infty,Y}$.
\end{definition} 
We include here some examples that motivate this definition and that will be used later on. 

\begin{enumerate}
\item $Y(Q):=w(Q)$. This corresponds to the $A_\infty$ class of weights. See Section \ref{sec:Ainfty} for a more detailed discussion on this subject.\\
\item $Y(Q):=w(2Q)$. Associated to this functional is the so called weak $A_\infty$ class of weights and will be treated on Section \ref{sec:WeakAinfty}.\\
\item $Y(Q):= \displaystyle\int_{\RN}M(\chi_{\strt{1.7ex}Q} )^pw$.  Associated to the $C_p$ class of weights, see Section \ref{sec:Cp}.\\
\item $Y(Q)=w_r(Q):=|Q|\,\displaystyle\left( \frac{1}{|Q|}\int_{Q }w^{r}dx\right)^{1/r}$, $1<r<\infty$. Also related to the $A_\infty$ class which  produces more precise estimates as shown in Corollary \ref{AinftyBMO}.

\end{enumerate}

\begin{theorem}  \label{CharactAinfty}
Let $v$ be a weight. There exist some dimensional constants $c_n, C_n$ independent of $v$ such that
\begin{equation}\label{eq:BMO-Ainfty-twoinequalituies}
c_n \ [v]_{A_{\infty,Y}} \le \sup_{b: \Norm{b}{\BMO}=1}\|b\|_{\BMO_{v,Y}} \le C_n \ [v]_{A_{\infty,Y}}.
\end{equation}
In other words, we have that
\begin{equation*} 
[v]_{A_{\infty,Y}} \approx \sup_{b: \Norm{b}{\BMO}=1}\sup_{Q} \frac{1}{Y(Q)}\int_Q|b(x)-b_{\strt{1.7ex}Q}| \,vdx
\end{equation*}
\end{theorem}
We will present some interesting corollaries of this result in  Section \ref{Sec:Cor}.
We observe that no condition is assumed on the functional $Y:\mathcal{Q}\to (0,\infty)$ nor on the weight in the theorem above. However, for the next theorem we need to restrict ourselves to a special class of functionals $Y$. Before introducing this special class of functionals we recall the notion of $L$-smallness that was introduced in \cite{PR-Poincare} within the context of generalized Poincar\'e inequalities.
\begin{definition}\label{def:L-small}
We say that a family of pairwise disjoint subcubes $\{Q_i\}$ contained in a cube $Q$ is $L$-small if 
\begin{equation}\label{eq:L-small}
\sum_{i}|Q_i|\le \frac{|Q|}{L}.
\end{equation}
In that case we say that $\{Q_i\}\in S_Q(L)$ or if the cube $Q$ is clear by the context that $\{Q_i\}\in S(L)$.
\end{definition} 
This condition arises tipically when considering the Calder\'on-Zygmund decomposition of the level sets of a non-negative function $f$ in a given cube $Q$ at level $L>1$ under the assumption $\avgint_Q f=1$. 

We define now the following  class of functionals.
\begin{definition}\label{def:Y}
We say that a functional $Y\in \mathcal{Y}_q$ for $q>1$ if there exists $c>0$ such that for any cube $Q$ and any family $\Lambda$ of pairwise disjoint subcubes of $Q$ with $\Lambda\in S(L)$, the following inequality holds
\begin{equation}\label{DefY(Q)}
\sum_{P\in\Lambda} Y(P) 
\leq  c\, Y(Q)   \left (\frac{1}{L}\right )^{\frac{1}{q}},
\end{equation}
and we will denote by $\beta_{\strt{1.7ex}Y}$ the smallest of the  constants $c$. 
\end{definition}

\begin{remark}\label{def wr}
Our model example in this class is given by
\begin{equation}\label{wr}
Y(Q)=w_r(Q):=\left( \frac{1}{|Q|}\int_{Q }w^{r}dx\right)^{1/r} \,|Q|= \left(\int_{Q}w^{r}dx\right)^{1/r} \,|Q|^{1/r'}
\end{equation}
where $r>1$.  H\"older's inequality yields \eqref{DefY(Q)} with  constant $\beta_{\strt{1.7ex}Y}\leq 1$ and exponent $q=r'$. 
\end{remark}
We can now state the following theorem.
\begin{theorem} \label{thm:GenAssymBMO}  Let $w$  be a 
weight and let $Y\in \mathcal{Y}_q$, $q>1$, namely a functional satisfying \eqref{DefY(Q)}. Suppose further that there is a constant $c$ such that for any cube $Q$
$$
w(Q)\leq    Y(Q).
$$ 
Then if $f\in \BMO$, there is  dimensional constant $c$ such that  for each cube $Q$, 
\begin{equation*}
\left(\frac{1}{ Y(Q)  } \int_{ Q }   |f -f_{Q}|^p     \,wdx\right)^{1/p}\, \leq c\, p\,q\,  \beta_{\strt{1.7ex}Y}\|f\|_{\BMO}
\end{equation*}
\end{theorem}
We will show some consequences of this result in Section \ref{Sec:Cor}.
Our next result can be seen as an update version of the work \cite{MW} which is related to the $\BMO_{w}$ classes following our notation. Being more precise we obtain quantitative versions of the results in \cite{MW}.
\begin{theorem}\label{ThmBloomBMO}
Let $b\in \BMO_{1,w}$, namely $\sup_Q\frac{1}{w(Q)}\int_Q\abs{f-f_{Q} }<\infty$. 
\begin{enumerate}
\item If $w\in A_{1}$ we have that for every $q>1$, 
\[
\left(\frac{1}{w(Q)}\int_{Q}\left|\frac{b(x)-b_{Q}}{w}\right|^{q}w(x)dx\right)^{\frac{1}{q}}\leq c_{n}\|b\|_{\BMO_{1,w}}q[w]_{A_{1}}^{\frac{1}{q'}}[w]_{A_{\infty}}^{\frac{1}{q}},
\]
and hence for any cube $Q$
\begin{equation} \label{genJN}
\norm{ \frac{ f- f_{Q} }{w} }_{ \exp L(Q,w ) } \, \leq c\, [w]_{A_1}\, \|f\|_{\BMO_{1,w}}.
\end{equation}
\item If $w\in A_{p}$ then, 
\[
\left(\frac{1}{w(Q)}\int_{Q}\left|\frac{b(x)-b_{Q}}{w}\right|^{p'}w(x)dx\right)^{\frac{1}{p'}}\leq c_{n}p'\|b\|_{\BMO_{1,w}} [w]_{A_{p}}^{\frac{1}{p}}\,[w]_{A_{\infty}}^{\frac{1}{p'}}.
\]
\end{enumerate}
\end{theorem}

The remainder of the paper is organized as follows. Section \ref{Sec:Cor} is devoted to provide consequences of some of the main results and in Section \ref{sec:Proofs} we provide the proofs of the main results.

\section{Some applications and consequences of the main results}\label{Sec:Cor}

\subsection{The \texorpdfstring{$A_\infty$}{Ainfty} class}\label{sec:Ainfty}

Our first Corollary provides interesting information related to $A_\infty$ weights. In particular we will provide a new characterization of the class via Fujii-Wilson constant.

\begin{corollary} \label{AinftyBMO} Let $w$ be a weight. Then
\[
[w]_{A_{\infty}} \approx \sup_{b: \Norm{b}{\BMO}=1}\sup_{Q} \frac{1}{w(Q)}\int_Q|b(x)-b_Q| \,w(x)dx.
\]
\end{corollary}
Our second corollary allows us to reprove known John-Nirenberg type estimates.
\begin{corollary}
Let $w \in A_{\infty}$ and $f\in \BMO$. There exists a dimensional constant $c_n$ independent of $f$ and $w$ such that  for each cube $Q$, 
\begin{equation}\label{1part}
\left(\frac{1}{ w(Q)  } \int_{ Q }   |f -f_{Q}|^p     \,wdx\right)^{1/p}\, \leq c_n\, p\,[w]_{A_{\infty}} \|f\|_{\BMO}
\end{equation}
and hence 
\begin{equation}\label{2part}
\norm{f -f_{Q}}_{\exp L(Q, \frac{wdx}{w(Q)} )} \leq C_n\, [w]_{A_{\infty}} \|f\|_{\BMO}.
\end{equation}
\end{corollary}

For the proof  of the latter we consider the functional $Y(Q)=w_r(Q)=\left( \frac{1}{|Q|}\int_{Q }w^{r}dx\right)^{1/r} \,|Q|$\,  which satisfies Definition \ref{def:Y} 
with  constants $\beta_{\strt{1.7ex}Y}\leq 1$ and exponent $q=r'$ by Remark \ref{def wr}.  Hence by Theorem \ref{thm:GenAssymBMO}
\begin{equation*}
\left(\frac{1}{ w_r(Q)  } \int_{ Q }   |f -f_{Q}|^p     \,wdx\right)^{1/p}\, \leq c_n\, p\,r'\,  \|f\|_{\BMO}
\end{equation*}
Finally, if  $w \in A_{\infty}$, the reverse H\"older inequality \eqref{RHI} with $r=r(w)$  so that  $r' \approx [w]_{A_{\infty}}$, yields \eqref{1part}.  
Inequality \eqref{2part}  will be obtained from the following well known measure theory argument. Consider a probability space $(X,\mu)$ and a function $g$  such that for some $p_0\geq 1$, $c>0$, and $\alpha>0$ we have that 
$$
\norm{g}_{L^p(X,\mu)} \leq c\,p^{\alpha} \qquad p\geq p_0.
$$
Then for a universal multiple of $c$,  
$$
\norm{g}_{\exp L^{\frac{1}{\alpha}} (X,\mu)} \leq c.\,
$$
We conclude the section by mentioning that above argument cannot be so precise if we use the more natural functional $Y(Q)=w(Q)$ instead of the functional $Y(Q)=w_r(Q)$. Indeed, if we check the details we would get an exponential growth $e^{c[w]_{A_{\infty}}}$  both in \eqref{1part} and in \eqref{2part}  instead of linear $c[w]_{A_{\infty}}$.

\subsection{The weak \texorpdfstring{$A_\infty$}{Ainfty} class}\label{sec:WeakAinfty}

Besides the example of functional displayed in \eqref{wr} another case of interest is the functional defined by the \emph{weak} condition:
\[Y(Q)=w(2Q).\]

This is related to the condition introduced by 
E. Sawyer in \cite{Sawyer-1981} by defining the  ``weak'' $A_{\infty}$ class as those weights satisfying the estimate
$$
w(E)\leq c\,\left(\frac{|E|}{|Q|}\right)^{\delta} w(2Q).
$$
This class of weights is very interesting since appears in many contexts  like the theory of quasiregular mappings or regularity fot solutions of elliptic PDE's (see, for example, \cite{BojIwa}).

In \cite{AHT}, the weak $A_\infty$ class was characterized by means of a suitable Fujii-Wilson type $A_\infty$ constant, namely
\begin{equation*}
  [w]^{weak}_{A_\infty}:=\sup_Q\frac{1}{w(2Q)}\int_Q M(w\chi_{\strt{1.7ex}Q} ).
\end{equation*}
Notice that the constant $2$ in the average could be replaced by any parameter $\sigma>1$ as shown in \cite{AHT}. It is also shown  there that the following reverse Holder's inequality holds
\begin{equation} \label{SharpWeakRHI}
  \Big(\avgint_Q w^{r(w)}\Big)^{\frac{1}{r(w)} }\leq 2\avgint_{2Q} w.
\end{equation}
with $r(w):=1+\frac{1}{ \tau_n \,[w]^{weak}_{A_{\infty}}}$ where $\tau_n$  is a dimensional constant that we may take to be $\tau_n \approx 2^{n}$.

Relying upon Theorems \ref{CharactAinfty} and \ref{thm:GenAssymBMO} and arguing as in the preceding section we obtain the following corollaries
\begin{corollary} Let $w$ be weight. Then
\begin{equation*}
[w]^{weak}_{A_\infty} \approx \sup_{f: \Norm{f}{\BMO}=1}\sup_{Q} \frac{1}{w(2Q)}\int_Q|f(x)-f_Q| \,wdx.
\end{equation*}
\end{corollary}
\begin{corollary}
Let $w \in A^{weak}_{\infty}$, and $f\in \BMO$. There exists a dimensional constant $c_n$ independent of $f$ and $w$ such that  for each cube $Q$, 
\begin{equation*}
\left(\frac{1}{ w(2Q)  } \int_{ Q }   |f -f_{Q}|^p     \,wdx\right)^{1/p}\, \leq c_n\, p[w]^{weak}_{A_{\infty}} \|f\|_{\BMO}
\end{equation*}
and hence 
\begin{equation*}
\norm{f -f_{Q}}_{\exp L(Q, \frac{wdx}{w(2Q)} )} \leq C_n\, [w]^{weak}_{A_{\infty}} \|f\|_{\BMO}.
\end{equation*}
\end{corollary}

\subsection{The \texorpdfstring{$C_p$}{Cp} class}\label{sec:Cp}

The $C_p$ class is a class of weights containing the weak $A_\infty$ class considered above and hence larger than the $A_\infty$ class.  It is a very interesting class of weights which is related intimately to the theory of singular integrals although very recently in \cite{ABES} appeared some application to PDE.  
Indeed,  it is a well known fact that $w\in A_\infty$ is a sufficient condition for the so called Coifman-Fefferman estimate, namely 
\begin{equation}\label{CF}
\int_{\RN} T^*f(x)^pw(x)dx\leq c_{n,p,T,w}\int_{\RN} Mf(x)^pw(x)dx\qquad 0<p<\infty
\end{equation}
where $T^*$ is  the maximal singular integral operator associated  to the Calder\'on-Zygmund  singular integral operator $T$.  

Muckenhoupt \cite{MCF} proved that $w\in A_\infty$ is not necessary for \eqref{CF} to hold. He showed in the case of the Hilbert transform that for $1<p<\infty$, the correct  necessary condition is that $w\in C_p$, namely, that for every cube $Q$ and every measurable subset of $E\subset Q$, 
\begin{equation}\label{Cp}
w(E)\leq c\left(\frac{|E|}{|Q|}\right)^\delta \int_{\RN}M(\chi_{\strt{1.7ex}Q} )^pw
\end{equation}
Later on Sawyer \cite{SCF} showed that the $C_p$ condition is ``almost'' sufficient in the sense  that  $w\in C_{p+\varepsilon}$ implies \eqref{CF} for $1<p<\infty$. A natural counterpart for the classical Fefferman-Stein estimate relating $M$ and $M^\#$,  the classical sharp maximal function, was provided by Yabuta \cite{YCF} and slightly improved by Lerner in \cite{LCF}. Recently, in \cite{CLPRCF} Sawyer's result was extended to the full range $0<p<\infty$ imposing as a sufficient condition that $w\in C_{\max\{1,p\}+\varepsilon}$ and including estimates for other operators as well. Similar estimates were settled too for the weak norm relying upon sparse domination.  At this point,  note that whether $C_p$ is sufficient for \eqref{CF} to hold remains an open question.

A fact that makes   the $C_p$ classes interesting is that weights in those classes are allowed to have ``holes'', namely to be zero in sets of no null measure, in some reasonable sense. 
Examples of those kind of weights can be found in \cite{MCF,BCF}.

Very recently Canto \cite{CCF} settled a suitable quantitative reverse H\"older inequality for $C_p$ classes in terms of the following Fujii-Wilson type constant
\[ [w]_{\strt{1.7ex} C_p}:=\sup_Q\frac{1}{\int_{\RN}M(\chi_{\strt{1.7ex}Q} )^pw}\int_Q M(w\chi_{\strt{1.7ex}Q} ).\]
Relying upon that reverse H\"older inequality it was also established in \cite{CCF} that if $w\in C_q$ for $1<p<q<\infty$ then
\[\|T^*f\|_{\strt{1.7ex}L^p(w)}\lesssim[w]_{\strt{1.7ex} C_q}\log\left(e+[w]_{\strt{1.7ex} C_q}\right)\|Mf\|_{\strt{1.7ex}L^p(w)}\]

Our next Corollary, which is again a direct consequence of Theorem \ref{CharactAinfty}, provides a new characterization of  the $C_p$ class.

\begin{corollary} Let $w$ a weight and $p>1$. We have that 
\begin{equation*}
[w]_{\strt{1.7ex}C_p} \approx \sup_{f: \Norm{f}{\BMO}=1}\sup_{Q} \frac{1}{\int_{\RN}M(\chi_{\strt{1.7ex}Q} )^pw}\int_Q|f(x)-f_Q| \,wdx
\end{equation*}
\end{corollary}

\section{Proofs of the main results}\label{sec:Proofs}
\subsection{Proof of Theorem \ref{CharactAinfty} } \label{sec:CharactAinfty}
We recall that a family of cubes $\mathcal{S}$ is $\eta$-sparse if for every $Q\in\mathcal{S}$ there exists a measurable subset $E_Q\subset Q$ such that 
\begin{enumerate}
\item $\eta |Q|\leq |E_Q|$.
\item The sets $E_Q$ are pairwise disjoint. 
\end{enumerate}

We will start by showing the second inequality from \eqref{eq:BMO-Ainfty-twoinequalituies}. We will rely upon a simplified version of  \cite[Lemma 5.1]{LORRAdv} (see also \cite{HComm2}).

\begin{lemma}\label{LemBabyLerner} Let $Q$ a cube. There exists a sparse family $\mathcal{S}\subset \mathcal{D}(Q)$, where $\mathcal{D}(Q)$ stands for the dyadic grid relative to $Q$, such that 
\[
|b(x)-b_{Q}|\chi_{\strt{1.7ex}Q} (x) \leq c_n\sum_{P\in\mathcal{S}} \frac{1}{|P|}\int_P  |b(y)-b_{P}|\,dy \chi_{P}(x).
\]
\end{lemma}

Armed with that lemma we can argue as follows:
\begin{eqnarray*}
\frac{1}{Y(Q)}\int_{Q}|b-b_{Q}|v  & \le &  c_n \frac{1}{Y(Q)}\int_Q \sum_P\left (\frac{1}{|P|}\int_P|b(y)-b_P|dy\right ) \chi_{\strt{1.7ex}P} (x) v(x) dx\\
&\le & c_n \frac{1}{Y(Q)}\sum_P\left (\frac{1}{|P|}\int_P|b(y)-b_P|dy\right ) v(P)
\end{eqnarray*}

Now, since $b\in BMO$, we have that
\begin{eqnarray*}
\frac{1}{Y(Q)}\int_{Q}|b-b_{Q}|v  & \le &  C_n \frac{1}{Y(Q)} \sum_P\left (\frac{1}{|P|}\int_P|b(y)-b_P|dy\right )v(P)\\
 & \le & C_n\|b\|_{\BMO}\frac{1}{Y(Q)}\sum_{P}v(P)\\
& \lesssim &\|b\|_{\BMO}\frac{1}{Y(Q)}\sum_{P\in\mathcal{S}}\frac{v(P)}{|P|}|E_{P}|\\
&\leq&\|b\|_{\BMO}\frac{1}{Y(Q)}\sum_{P\in\mathcal{S}}\int_{E_{P}}M(v\chi_{\strt{1.7ex}Q})\\
& \leq & \|b\|_{\BMO}\frac{1}{Y(Q)}\int_{Q}M(v\chi_{\strt{1.7ex}Q})\leq\|b\|_{\BMO}[v]_{A_{\infty,Y}}.
\end{eqnarray*}
This proves that
\[ \|b\|_{\BMO_{v,Y}}\lesssim \|b\|_{\BMO}\,[v]_{A_{\infty,Y}} \]

For the second part of the proof, let us denote
\[X:= \sup_{\|b\|_{\BMO}=1}\|b\|_{\BMO_{v,Y}}.\]
Then it suffices to show that $[v]_{\strt{1.7ex}A_{\infty,Y}}\lesssim_n X$ and we shall
assume that $X<\infty$ since otherwise there's nothing to prove. We note as well that from the definition of $X$ it follows that 
\begin{equation}\label{eq:hyp}
\|b\|_{\strt{1.7ex}\BMO_{v,Y}}\leq X\,\|b\|_{\strt{1.7ex}\BMO}.
\end{equation}

Taking this into account, we claim first that for every cube $Q$
\begin{equation}
v(Q)\leq4  X\,Y(Q).\label{eq:uv}
\end{equation}
We may assume first that $v(Q)>0$, otherwise is trivial.  Now, for a cube $Q$ we let $\tilde{Q}\subset Q$ be another cube
such that $\frac{1}{2}|Q|=|\tilde{Q}|$. Then $(\chi_{\strt{2ex}\tilde{Q}})_{\strt{1.7ex}Q}=\frac{1}{2}$
and we have that 
\[
\begin{split}\frac{v(Q)}{Y(Q)} & =\frac{2}{Y(Q)}\int_{Q}|\chi_{\tilde{Q}}-(\chi_{\tilde{Q}})_{Q}|v\leq 2\|\chi_{\tilde{Q}}\|_{\BMO_{v,Y}}\\
 & \leq 2X\|\chi_{\tilde{Q}}\|_{\BMO}\leq 4X,
\end{split}
\]
from which the claim follows.

We now observe that to prove the theorem it suffices to show that
there exists some finite constant $\alpha_{n}\geq1$ such that for every
cube $Q$, the inequality

\begin{equation}
\frac{1}{|Q|}\int_{Q}\log^{+}\left(\frac{v(x)}{\alpha_{n}^{2}v_{Q}}\right)v(x)dx\ \lesssim_n X\,\frac{Y(Q)}{|Q|}\label{eq:Target1}
\end{equation}
holds. Indeed, provided (\ref{eq:Target1}) holds, and taking into
account (\ref{eq:uv}),
\[
\begin{split}\frac{1}{|Q|}\int_{Q}M(\chi_{Q}v)dx & \simeq\frac{1}{|Q|}\int_{Q}\left(1+\log^{+}\left(\frac{v(x)}{v_{Q}}\right)\right)v(x)dx\\
 & =\frac{v(Q)}{|Q|}+\frac{1}{|Q|}\int_{Q}\log^{+}\left(\frac{\alpha_{n}^{2}v(x)}{\alpha_{n}^{2}v_{Q}}\right)v(x)dx\\
 & \leq3\frac{v(Q)}{|Q|}\log^{+}(\alpha_{n})+\frac{1}{|Q|}\int_{Q}\log^{+}\left(\frac{v(x)}{\alpha_{n}^{2}v_{Q}}\right)v(x)dx\\
 & \leq3\log^{+}(\alpha_{n})\frac{v(Q)}{|Q|}+\nu_{n}\,X\frac{Y(Q)}{|Q|}\\
 & \lesssim_n\frac{Y(Q)}{|Q|}X,
\end{split}
\]
and this yields 
\[
\frac{1}{Y(Q)}\int_{Q}M(\chi_{Q}v)dx\lesssim_n X.
\]
Now, to prove \eqref{eq:Target1} we note first that
\[
\frac{1}{|Q|}\int_{Q}\log^{+}\left(\frac{v(x)}{\alpha_{n}^{2}v_{Q}}\right)v(x)dx=\frac{1}{|Q|}\int_{L_{Q}}\log^{+}\left(\frac{v(x)}{\alpha_{n}^{2}v_{Q}}\right)v(x)dx
\]
where $L_{Q}=$ $\{x\in Q: v(x)\geq \alpha_{n}^{2}\,v_{Q}\}$ and $\alpha_n$ is to be chosen. Hence, settling
(\ref{eq:Target1}) is equivalent to prove that 

\begin{equation}\label{Ainfty}
\frac{1}{|Q|}\int_{L_{Q}}\log^{+}\left(\frac{v(x)}{\alpha_{n}^{2}v_{Q}}\right)v(x)dx\ \leq\gamma_{n}X\frac{Y(Q)}{|Q|} 
\end{equation}
for appropriate constants $\alpha_{n},\gamma_{n}\geq1$. 

Now, recall that if $w\in A_{1}$ then $\log(w)\in\BMO$. Furthermore, tracking the dependence on the $A_1$ constant in \cite[Theorem 3.3 p. 157]{GCRdF}
\[
\|\log(w)\|_{\strt{1.7ex} \BMO}\leq  2\log([w]_{A_{1}}).
\]
Hence, in particular, if we choose 
$$w=M\left(\frac{v\chi_{Q}}{v_Q}\right)^{1/2},$$
it is well known (see for instance \cite[Theorem 3.4 p. 158]{GCRdF}) that $w\in A_{1}$ with 
$$[w]_{A_{1}}\leq d_{n}.$$
Hence,  if we let 
$$b=\log w$$ 
there exists a constant a dimensional constant $\rho_{n}$, i.e. independent of $Q$, such that 
\[
\left\Vert b \right\Vert _{\BMO}\leq\rho_{n}.
\]
Observe that although $b$ depends on $Q$ its BMO constant is just dimensional. Combining this estimate with \eqref{eq:hyp}

\begin{eqnarray}
 &  & \frac{1}{Y(Q)}\int_{Q}\left|b  -  b_{Q}\right|v(x)dx\label{bmo}
\leq X\left\Vert b  \right\Vert _{\BMO}\leq X\rho_{n},\nonumber 
\end{eqnarray}
namely 
\[
\avgint_{Q} \left|b  -  b_{Q}\right| v(x)dx\leq X\rho_{n}\frac{Y(Q)}{|Q|}.
\]
In view of the preceding estimate, to settle (\ref{Ainfty}), and
hence ending the proof, it suffices to check that for every $x\in L_{Q}$
for $\alpha_{n}>1$ to be chosen we have that
\begin{equation}\label{claim}
\left|b(x)  -  b_{Q}\right| \geq \frac{1}{2}\log^{+}\left(\frac{v(x)}{\alpha^2_{n}v_{Q}}\right)
\end{equation}
To verify this we observe first that combining Jensen's inequality and Kolmogorov's inequality with dimensional constant $c_n=2\|M\|^{\frac{1}{2}}_{L^1(\mathbb{R}^n)\rightarrow L^{1,\infty}(\mathbb{R}^n)}$  (see \cite[Ex. 2.1.5 p. 100]{GrafakosCF3rd}), we have
$$b_{Q} = \avgint_{Q} \log w = \avgint_{Q} \log \left(\frac{M(v\chi_{Q})}{v_Q} \right)^{1/2} \leq 
\log \left[\avgint_{Q} \left(\frac{M(v\chi_{Q})}{v_Q} \right)^{1/2} \right]\leq \log c_n. 
$$ 
If we further assume that $x\in L_{Q}$, namely that $v(x)\geq \alpha_n^2 v_Q$, where $\alpha_n$  is yet to be chosen, we have 
$$b_{Q}\leq \log c_n \leq \log \left(c_n \frac{v(x)^{1/2}}{\alpha_n (v_Q)^{1/2}}\right) \leq \log \left(\frac{M(v\chi_Q)(x)^{1/2}}{ (v_Q)^{1/2}}\right)\leq  \log w=b
$$ 
choosing $\alpha_n = c_n$. Hence, for these $x\in L_{Q}$
\[
\begin{split}
\left|b(x)  -  b_{Q}\right| &= b(x)  -  b_{Q} \geq b(x)  - \log c_n = \log \left(\frac{w(x)}{c_n}\right)\\
&= \frac{1}{2} \log\left[ \frac{1}{c^2_n}\, \left(\frac{M(v\chi_{Q})}{v_{Q}}\right)\right]
\geq \frac{1}{2} \log\left[ \frac{1}{c^2_n}\, \left(\frac{v(x)}{v_{Q}}\right)\right]. 
\end{split}
\]
This ends the proof of \eqref{claim} and hence the proof of \eqref{Ainfty} with $\alpha_{n}=c_{n}$
and $\gamma_{n}=2\rho_{n}$.

\subsection{Proof of Theorem  \ref{thm:GenAssymBMO}} 
In this section we will present two different approaches for Theorem \ref{thm:GenAssymBMO}. One of them is based in some new ideas from \cite{PR-Poincare} involving self improving properties for smallness preserving functionals related to generalized Poincar\'e inequalities. This was in fact inspired by the  proof of John-Nirenberg's lemma given in \cite{Journe}. Relying upon this approach we will establish the following inequality

\begin{equation}\label{eq:P(p,p)-CZ}
\left( \frac{1}{ Y(Q)  } \int_{ Q }   |f -f_{Q}|^p     \,wdx\right)^{\frac{1}p}  \, \leq  c_n\, p q\,\max\left\{1, \beta_Y^{q}\right\}\|f\|_{\BMO},
\end{equation}
Even though that the estimate above doesn't provide the best dependence on $\beta_Y$ we have included its proof for the sake of the interest of the approach used.

The other proof that we will present here can be seen as a very interesting application of the so called \emph{sparse approach} for studying singular integrals. 
We remark that in this BMO type estimate this approach provides a linear bound in terms of  $\beta_Y$, namely, we will prove that
\begin{equation*}
\left( \frac{1}{ Y(Q)  } \int_{ Q }   |f -f_{Q}|^p     \,wdx\right)^{\frac{1}p}  \, \leq  c_n\, p q\,\beta_Y\|f\|_{\BMO}.
\end{equation*}
Unfortunately, this method does not work so precisely in the general scenario of generalized Poincar\'e inequalities   

\subsubsection{Proof based on the smallness property}

As we announced above, in this section we will settle \eqref{eq:P(p,p)-CZ}. By homogeneity we may assume that $\|f\|_{\BMO}=1$,
\begin{equation}
\frac{1}{|Q|}\int_{Q} |f-f_{Q}| \le 1. \label{eq:UnWeightedStartingPointL1}
\end{equation}

We may assume that $f$ is bounded. 
Fixed  one cube $Q$.  We can consider the local Calder\'on-Zygmund decomposition  of $|f-f_{Q}|$ relative to $Q$ at level $L$ on $Q$ for a large universal 
constant $L>1$ to be chosen. Let $\mathcal{D}(Q)$ be the family of dyadic subcubes of $Q$. The Calder\'on-Zygmund (C-Z) decomposition yields a collection $\{Q_{j}\}$ of cubes such that $Q_j\in \mathcal{D}(Q)$, maximal with respect to inclusion, satisfying

\begin{equation}\label{eq:CZ1}
L  < \frac{1}{|Q_{j}|  }\int_{Q_{j}} |f-f_{Q}| \, dy. 
\end{equation}
Then, if $P$ is dyadic with $P \supset Q_j$ 
\begin{equation}\label{eq:CZ2}
\frac{ 1 }{|P|  }
\int_{P} |f-f_{Q}| \, dy   \leq L
\end{equation}
and hence 
\begin{equation}\label{eq:CZ3}
L  < \frac{ 1 }{|Q_{j}|  }
\int_{Q_{j}} |f-f_{Q}| \, dy \leq L\,2^{n}
\end{equation}
for each integer $j$.  Also note that 
$$  \left \{x\in Q: M_Q^d\left ( |f-f_{Q}|\chi_{  Q }\right )(x) > L   \right \} = \bigcup_{j}Q_j=:\Omega_L
$$
where $M^d_Q$ stands for the dyadic maximal function adapted to the cube $Q$. That is, 
$$
M^d_Q(f)(x):=\sup_{P\ni x}\avgint |f(y)|\ dy\qquad x\in Q, P\in \mathcal{D}(Q)
$$
Then, by the Lebesgue differentiation theorem it follows that
$$|f-f_{Q}| \leq L   \qquad              a.e. \ x \notin \Omega_L
$$

Also, observe that by \eqref{eq:CZ1} (or the weak type $(1,1)$ property of $M$) and recalling our starting assumption \eqref{eq:UnWeightedStartingPointL1}, we have that $\{Q_i\}\in S(L)$, namely
\begin{equation*}\label{eq:CZ4a}
|\Omega_L|=\left|\bigcup_{j}Q_j \right| < \frac{|Q|}{L}.
\end{equation*}

Now, given the C-Z decomposition of the cube $Q$, we perform the classical C-Z of the function $f-f_{Q}$ as
\begin{equation}\label{eq:CZ-f-fQ}
f-f_{Q}=g_Q+b_Q,
\end{equation}
where the functions $g_Q$ and $b_Q$ are defined as usual. We have that

\begin{equation}\label{eq:gQ}
g_Q(x) = \left \{
\begin{array}{ccc}
f-f_{Q}, &  &  x \notin \Omega_L \\
&&\\
\displaystyle \avgint_{Q_i}(f-f_{Q}), &  &  x\in \Omega_L,  x\in Q_i
\end{array}
\right .
\end{equation}
Note that this definition makes sense since the cubes $\{Q_i\}$ are disjoint, so any $x\in \Omega_L$ belongs to only one $Q_i$. Also note that condition \eqref{eq:CZ3} implies that 
\begin{equation}\label{eq:gQ-bounded}
g_Q(x)\le 2^nL 
\end{equation}
for almost all $x\in Q$. The function $b_Q$ is determined by this choice of $g_Q$ as the difference 
\begin{equation*}
b_Q= f-f_{Q} -g_Q,
\end{equation*}
but we  also have a representation as
\begin{equation}\label{eq:bQ}
b_Q(x)=\sum_i \left (f(x)-f_{Q_i}\right )\chi_{Q_i}(x)=\sum_i b_{Q_i},
\end{equation}
where $b_{Q_i}=(f(x)-f_{Q_i})\chi_{Q_i}(x)$.

Now we start with the estimation of the desired $L^p$ norm from \eqref{eq:P(p,p)-CZ}. Consider on $Q$ the measure $\mu$ defined by $d\mu=\frac{w }{Y(Q)}\chi_{Q}$. Then, by the triangle inequality, we have
\begin{eqnarray*}
\left( \frac{1}{ Y(Q)  } \int_{ Q }   |f -f_{Q}|^p   \,wdx\right)^{\frac{1}{p} } & \le & \|g_Q\|_{L^p(\mu)}+\|b_Q\|_{L^p(\mu)}\\
& \le & 2^nL + \left (\dfrac{1}{Y(Q)}\int_{\Omega_L}\sum_j |b_{Q_j}|^p \, w dx \right )^{1/p}\\
\end{eqnarray*}
since  we assume $ w(Q)\leq    Y(Q).$

Let us observe that the last integral of the sum, by the localization properties of the functions $b_{Q_i}$, can be controlled:
\begin{eqnarray*}
\int_{\Omega_L} |\sum_j b_{Q_j} |^p \, w dx  & \le & \sum_j\int_{Q_j}\left | b_{Q_j}\right |^p \, w dx\\
& = & \sum_j\frac{Y(Q_j)}{Y(Q_j)}\int_{Q_j}\left |f-f_{Q_j}\right |^p \, w dx \\
& \le & X^p \sum_j  Y(Q_j),
\end{eqnarray*}
where $X$ is the quantity defined by
$$
X=\sup_Q \left (\frac{1}{Y(Q)}\int_{Q} |f-f_{Q}|^p \, w dx \right )^{1/p},
$$
which is finite since we are assuming that  $f$ is bounded.
Then we obtain that
\begin{eqnarray*}
\left( \frac{1}{ Y(Q)  } \int_{ Q }  |f -f_{Q}|^p  \,wdx\right)^{\frac{1}{p} } & \le & 2^nL + X\left ( \dfrac{\sum_i Y(Q_i) }{Y(Q)}\right )^{1/p}\\
& \le &2^n L +  X\, (\frac{\beta_Y}{L^{1/q}})^{1/p}
\end{eqnarray*}
by the smallness preserving hypothesis. This holds for every cube $Q$, so taking the supremum we obtain
$$
X\le 2^n L +  (\frac{\beta_Y}{L^{1/q}})^{1/p} X.
$$
Now we choose $L=2e\max\left\{\beta_Y^{q},1\right\}$ so the above inequality yields
$$
X\leq 2^n 2e\max\left\{\beta_Y^{q},1\right\}\, \left ((2e)^{1/pq}\right )'\le e2^{n+2}\,pq\,\max\left\{\beta_Y^{q},1\right\},
$$
using that $\left ((2e)^{1/s}\right )'\le 2s$, $s>1$. This is the desired inequality \eqref{eq:P(p,p)-CZ}:
$$
\left( \frac{1}{ Y(Q)  } \int_{ Q } |f -f_{Q}|^p \,wdx\right)^{\frac{1}{p} }\le c_n pq \max\left\{\beta_Y^{q},1\right\}.
$$

\subsubsection{Proof based on the sparse approach}
First we need to prove that the smallness condition implies suitable sparse conditions for functionals.

\begin{lemma}
Let $q>1$ and $Y\in \mathcal{Y}_q$  be a functional. 
Let $\mathcal{F}\subset\mathcal{D}(Q)$ be a family
of cubes. If there exists $L>1$ such that for every $P\in\mathcal{F}$ 
\[
\sum_{R\in\mathcal{F},R\subsetneq P,R\text{ maximal}}|R|\leq\frac{1}{L}|P|
\]
then
\[
\sum_{P\in\mathcal{F}}Y(P)\leq\kappa Y(Q)
\]
where $\kappa=\beta_Y\sum_{k=0}^{\infty}\frac{1}{L^{\frac{k}{q}}}$
\end{lemma}

\begin{proof}
We observe that 
\[
\sum_{P\in\mathcal{F}}Y(P)=\sum_{k=0}^{\infty}\sum_{P\in\mathcal{F}_{k}}Y(P)
\]
where $\mathcal{F}_{0}=\{Q\}$, $\mathcal{F}_{k}=\{P\subsetneq R\,:\,P\in\mathcal{F},R\in\mathcal{F}_{k-1},P\text{ maximal }\}$.
We note that, taking into account the properties of $\mathcal{F}$,
\[
\sum_{P_{j}^{1}\in\mathcal{F}_{1}}|P_{j}^{1}|\leq\frac{1}{L}|Q|
\]
\[
\sum_{P_{j}^{2}\in\mathcal{F}_{2}}|P_{j}^{2}|\leq\frac{1}{L}\sum_{P_{j}^{1}\in\mathcal{F}_{1}}|P_{j}^{1}|\leq\frac{1}{L^{2}}|Q|
\]
and in general 
\[
\sum_{P_{j}^{k}\in\mathcal{F}_{k}}|P_{j}^{k}|\leq\frac{1}{L^{k}}|Q|.
\]
Then, since $Y\in \mathcal{Y}_q$,
\[
\sum_{k=0}^{\infty}\sum_{P\in\mathcal{F}_{k}}Y(P)\leq \beta_Y\, \sum_{k=0}^{\infty}\left(\frac{1}{L^{k}}\right)^{\frac{1}{q}}Y(Q)
\]
and the desired estimate holds with $\kappa = \beta_Y\,\sum_{k=0}^{\infty}\frac{1}{L^{\frac{k}{q}}}.$
\end{proof}

We may assume, without loss of the generality that the sparse family in Lemma \ref{LemBabyLerner} satisfies a sparseness property as the one in the preceding Lemma with $L=2$.  We remit the reader to \cite[Section 6]{LNDy} for the equivalence between Carleson families, as the ones in the lemma, and sparse families and futher details. Taking that into account we can proceed as follows. Let $k$ be the only non-negative integer such that $k<p\leq k+1$. Then, since $\frac{w(Q)}{Y(Q)} \leq 1$
\begingroup
\allowdisplaybreaks
\begin{align*}
 & \left(\frac{1}{Y(Q)}\int_{Q}|f-f_{Q}|^{p}w\right)^{\frac{1}{p}}\leq \left(\frac{1}{Y(Q)}\int_{Q}|f-f_{Q}|^{k+1}w\right)^{\frac{1}{k+1}}\\
 & \leq c_{n}\left(\frac{1}{Y(Q)}\int_{Q}\left(\sum_{P\in\mathcal{F}}\frac{1}{|P|}\int_{P}|f-f_{P}|\chi_{P}(x)\right)^{k+1}w\right)^{\frac{1}{k+1}}\\
 & \leq c_{n}\|f\|_{\BMO}\left(\frac{1}{Y(Q)}\int_{Q}\left(\sum_{P\in\mathcal{F}}\chi_{P}(x)\right)^{k+1}w(x)dx\right)^{\frac{1}{k+1}}\\
 & \leq c_{n}\|f\|_{\BMO}\left((k+1)!\right)^{\frac{1}{p}}\left(\frac{1}{Y(Q)}\sum_{\stackrel{{\scriptstyle P_{k+1}\subseteq P_{k}\subseteq\dots\subseteq P_{1}\subseteq Q}}{P_{i}\in\mathcal{F}}}\int_{Q}\chi_{P_{1}}\chi_{P_{2}}\cdots\chi_{P_{k}}\chi_{P_{k+1}}w\right)^{\frac{1}{k+1}}\\
 & \leq c_{n}\|f\|_{\BMO}\left((k+1)!\right)^{\frac{1}{k+1}}\left(\frac{1}{Y(Q)}\sum_{\stackrel{{\scriptstyle P_{k+1}\subseteq P_{k}\subseteq\dots\subseteq P_{1}\subseteq Q}}{P_{i}\in\mathcal{F}}}w(P_{k+1})\right)^{\frac{1}{k+1}}\\
 & \leq c_{n}\|f\|_{\BMO}\left((k+1)!\right)^{\frac{1}{k+1}}\left(\frac{1}{Y(Q)}\sum_{\stackrel{{\scriptstyle P_{k+1}\subseteq P_{k}\subseteq\dots\subseteq P_{1}\subseteq Q}}{P_{i}\in\mathcal{F}}}Y(P_{k+1})\right)^{\frac{1}{k+1}}\\
 & \leq c_{n}\|f\|_{\BMO}\left((k+1)!\right)^{\frac{1}{k+1}}\left(\frac{1}{Y(Q)}\kappa_{q}\sum_{\stackrel{{\scriptstyle P_{k}\subseteq P_{k-1}\subseteq\dots\subseteq P_{1}\subseteq Q}}{P_{i}\in\mathcal{F}}}Y(P_{k})\right)^{\frac{1}{k+1}}\\
 & \dots\\
 & \leq c_{n}\|f\|_{\BMO}\left((k+1)!\right)^{\frac{1}{k+1}}\left(\frac{1}{Y(Q)}\kappa_{q}^{k+1}Y(Q)\right)^{\frac{1}{k+1}}\\
 & \leq c_{n}\|f\|_{\BMO}\left((k+1)!\right)^{\frac{1}{k+1}}\kappa_{q}
\end{align*}
\endgroup
where $\kappa_{q}=\beta_Y\sum_{k=0}^{\infty}\frac{1}{2^{k/q}}$.  Finally, a simple computation, using elementary Stirling estimates for $(k+1)!$, shows that this final inequality yields the desired estimate: 
\begin{equation*}
\left(\frac{1}{Y(Q)}\int_{Q}|f-f_{Q}|^{p}w\right)^{\frac{1}{p}} \le c_n p\,q\,\beta_Y\,\|f\|_{\BMO}.
\end{equation*}

\section{Proof of Theorem \ref{ThmBloomBMO}}

We rely upon  Lemma \ref{LemBabyLerner} to provide the proof of Theorem \ref{ThmBloomBMO}.
\begin{proof}
We start with the second estimate
\begingroup
\allowdisplaybreaks
\begin{align*} 
& \left(\frac{1}{w(Q)}\int_{Q}\left|\frac{b(x)-b_{Q}}{w}\right|^{p'}w(x)dx\right)^{\frac{1}{p'}}\\
 & \leq2^{n+2}\left(\frac{1}{w(Q)}\int_{Q}\left(\sum_{P\in\mathcal{S},P\subseteq Q}\frac{w(P)}{|P|}\frac{1}{w(P)}\int_{P}|b-b_{P}|\chi_{P}(x)\right)^{p'}\sigma(x)dx\right)^{\frac{1}{p'}}\\
 & \leq2^{n+2}\|b\|_{\BMO_{1,w}}\left(\frac{1}{w(Q)}\int_{Q}\left(\sum_{P\in\mathcal{S},P\subseteq Q}\frac{w(P)}{|P|}\chi_{P}(x)\right)^{p'}\sigma(x)dx\right)^{\frac{1}{p'}}\\
 & \leq c_{n}\|b\|_{\BMO_{1,w}}\left(\frac{(k+1)!}{w(Q)}\sum_{\stackrel{\scriptstyle{P_{i}\in\mathcal{S}}}{P_{k+1}\subseteq P_{k}\subseteq\dots\subseteq P_{1}\subseteq Q}}\left(\prod_{j=1}^k\langle w\rangle_{P_{j}}\right)\left(\langle w\rangle_{P_{k+1}}\right)^{p'-k}\sigma(P_{k+1})\right)^{\frac{1}{p'}}
\end{align*}
\endgroup
taking into account \cite[Lemma 5.1]{HRem} and choosing $k$ to be the unique integer such that $k\leq p'<k+1$.  A careful tracking of the constants involved in the argument in \cite[p. 102]{HRem} allows us to obtain the
desired result.

For the first estimate it suffices to notice that if $q>1$ then,  $[w]_{A_{q'}}\leq [w]_{A_1}$. Then using the computations above we are done.
\end{proof}

\bibliographystyle{amsalpha}

\begin{thebibliography}{GCHST91}

\bibitem[ABES]{ABES}
P. Auscher, S. Bortz, M.Egert, and O. Saari, \emph{Non-local ghering lemmas},
  Preprint, arXiv:1707.02080v1 (2018).

\bibitem[AHT17]{AHT}
T.~C. Anderson, T. Hyt\"{o}nen, and O. Tapiola, \emph{Weak
  {$A_\infty$} weights and weak reverse {H}\"{o}lder property in a space of
  homogeneous type}, J. Geom. Anal. \textbf{27} (2017), no.~1, 95--119.

\bibitem[AMPRR]{AMPRR}
N. Accomazzo, J. Mart\'inez~Perales, and I. P. Rivera~R\'ios,
  \emph{On {B}loom type estimates for iterated commutators of fractional
  integrals}, To appear in Indiana Univ. Math. Journal.

\bibitem[BI87]{BojIwa}
B. Bojarski and T.~Iwaniec, \emph{{$p$}-harmonic equation and quasiregular
  mappings}, Partial differential equations ({W}arsaw, 1984), Banach Center
  Publ., vol.~19, PWN, Warsaw, 1987, pp.~25--38.

\bibitem[Blo85]{Bloom}
S. Bloom, \emph{A commutator theorem and weighted {BMO}}, Trans. Amer.
  Math. Soc. \textbf{292} (1985), no.~1, 103--122.

\bibitem[Buc90]{BCF}
S. M. Buckley, \emph{Harmonic analysis on weighted spaces}, ProQuest
  LLC, Ann Arbor, MI, 1990, Thesis (Ph.D.)--The University of Chicago.
  \MR{2611907}

\bibitem[{Can}18]{CCF}
J. Canto, \emph{{Quantitative $C_p$ estimates for Calder\'on-Zygmund
  operators}}, arXiv e-prints (2018), arXiv:1811.05209.

\bibitem[CLPR]{CLPRCF}
M.~E. {Cejas}, K. {Li}, C. {Perez}, and I.~P.
  {Rivera-Rios}, \emph{{Vector-valued operators, optimal weighted estimates and
  the $C_p$ condition}}, to appear in Sci. China Math.

\bibitem[DMRO16]{DMRO-Ainfty}
J. Duoandikoetxea, F.~J. Mart{\'{\i}}n-Reyes, and S. Ombrosi,
  \emph{On the {$A_\infty$} conditions for general bases}, Math. Z.
  \textbf{282} (2016), no.~3-4, 955--972. \MR{3473651}

\bibitem[GC79]{GC}
J. Garc\'{i}a-Cuerva, \emph{Weighted {$H^{p}$} spaces}, Dissertationes
  Math. (Rozprawy Mat.) \textbf{162} (1979), 63. \MR{549091}

\bibitem[GCHST91]{GCHST}
J.~Garc\'{i}a-Cuerva, E.~Harboure, C.~Segovia, and J.~L. Torrea, \emph{Weighted
  norm inequalities for commutators of strongly singular integrals}, Indiana
  Univ. Math. J. \textbf{40} (1991), no.~4, 1397--1420.

\bibitem[GCRdF85]{GCRdF}
J. Garc{\'{\i}}a-Cuerva and J.~L. Rubio~de Francia, \emph{Weighted
  norm inequalities and related topics}, North-Holland Mathematics Studies,
  vol. 116, North-Holland Publishing Co., Amsterdam, 1985.

\bibitem[Gra14]{GrafakosCF3rd}
L. Grafakos, \emph{Classical {F}ourier analysis}, third ed., Graduate Texts
  in Mathematics, vol. 249, Springer, New York, 2014.

\bibitem[HLW17]{HLW}
I. Holmes, M.~T. Lacey, and .B~D. Wick, \emph{Commutators in the
  two-weight setting}, Math. Ann. \textbf{367} (2017), no.~1-2, 51--80.

\bibitem[HP13]{HP}
T. Hyt\"onen and C. P\'erez, \emph{Sharp weighted bounds involving
  ${A}_{\infty}$}, Anal. PDE \textbf{6} (2013), no.~4, 777--818.

\bibitem[Hyt]{HComm2}
T. Hyt\"{o}nen, \emph{The {$L^p$-to-$L^q$} boundedness of commutators with
  applications to the {J}acobian operator}.

\bibitem[Hyt14]{HRem}
T. Hyt\"{o}nen, \emph{The {$A_2$} theorem: remarks and complements},
  Harmonic analysis and partial differential equations, Contemp. Math., vol.
  612, Amer. Math. Soc., Providence, RI, 2014, pp.~91--106. \MR{3204859}

\bibitem[JN61]{JN}
F.~John and L.~Nirenberg, \emph{On functions of bounded mean oscillation},
  Comm. Pure Appl. Math. \textbf{14} (1961), 415--426.
  
\bibitem[Jou83]{Journe}
J. L. Journ{\'e}, \emph{Calder\'on-{Z}ygmund operators, pseudodifferential
  operators and the {C}auchy integral of {C}alder\'on}, Lecture Notes in
  Mathematics, vol. 994, Springer-Verlag, Berlin, 1983.


\bibitem[Ler10]{LCF}
A.~K. Lerner, \emph{Some remarks on the {F}efferman-{S}tein inequality}, J.
  Anal. Math. \textbf{112} (2010), 329--349. \MR{2763004}

\bibitem[LN17]{LNDy}
A.~K. {Lerner} and F. {Nazarov}, \emph{{Intuitive dyadic calculus: the
  basics}}, To appear in Expo. Math. (2017).

\bibitem[LORR17]{LORRAdv}
A.~K. Lerner, S. Ombrosi, and I.~P. Rivera-R\'{i}os, \emph{On
  pointwise and weighted estimates for commutators of {C}alder\'{o}n-{Z}ygmund
  operators}, Adv. Math. \textbf{319} (2017), 153--181.

\bibitem[LORR18]{LORR}
A.~K. Lerner, S. Ombrosi, and I.~P. Rivera-R\'{i}os, \emph{Commutators of singular integrals revisited}, Bulletin of the
  London Mathematical Society (2018).

\bibitem[Muc81]{MCF}
B.~Muckenhoupt, \emph{Norm inequalities relating the {H}ilbert transform to the
  {H}ardy-{L}ittlewood maximal function}, Functional analysis and approximation
  ({O}berwolfach, 1980), Internat. Ser. Numer. Math., vol.~60, Birkh\"auser,
  Basel-Boston, Mass., 1981, pp.~219--231. \MR{650277}

\bibitem[MW76]{MW}
B.Muckenhoupt and R.~L. Wheeden, \emph{Weighted bounded mean
  oscillation and the {H}ilbert transform}, Studia Math. \textbf{54} (1975/76),
  no.~3, 221--237.

\bibitem[PR]{PR-Poincare}
C. P{\'e}rez and E. Rela, \emph{Degenerate poincar\'e-sobolev
  inequalities}, To appear Trans. A.M.S.,
  \url{https://arxiv.org/abs/1805.10388}.

\bibitem[Saw82]{Sawyer-1981}
E.~T. Sawyer, \emph{Two weight norm inequalities for certain maximal and
  integral operators}, Harmonic analysis ({M}inneapolis, {M}inn., 1981),
  Lecture Notes in Math., vol. 908, Springer, Berlin-New York, 1982,
  pp.~102--127. \MR{654182}

\bibitem[Saw83]{SCF}
E.~T. Sawyer, \emph{Norm inequalities relating singular integrals and the
  maximal function}, Studia Math. \textbf{75} (1983), no.~3, 253--263.
  \MR{722250}

\bibitem[Yab90]{YCF}
K.~Yabuta, \emph{Sharp maximal function and {$C_p$} condition}, Arch. Math.
  (Basel) \textbf{55} (1990), no.~2, 151--155. \MR{1064382}

\end{thebibliography}

\end{document}